\documentclass[11pt]{amsart}

\usepackage{graphicx}
\usepackage[latin1]{inputenc}
\usepackage{amsmath,amsthm,amssymb,amsfonts}

\newtheorem{theorem}{Theorem}
\newtheorem{lemma}[theorem]{Lemma}

\newcommand{\R}         {\mathbb{R}}

\renewcommand{\le}         {\leqslant}

\renewcommand{\ge}         {\geqslant}
\renewcommand{\geq}         {\geqslant}
\renewcommand{\div}         {\,{\rm{div}}}
\newcommand{\Ric}         {\,{\rm{Ric}}}

\begin{document}

\title[Stable solutions on manifolds]{
Stable solutions of elliptic equations \\ on Riemannian 
manifolds}

\author{Alberto Farina, Yannick Sire
and Enrico Valdinoci}\thanks{
{\it AF}:
LAMFA -- CNRS UMR 6140 --
Universit\'e de Picardie Jules Verne --
Facult\'e de Math\'ematiques et d'Informatique --
33, rue Saint-Leu -- Amiens, France --
{\tt alberto.farina@u-picardie.fr} \\
{\it YS}:
Universit\'e Aix-Marseille 3, Paul C\'ezanne --
LATP --
Marseille, France --
{\tt sire@cmi.univ-mrs.fr}\\
{\it EV}:
Universit\`a di Roma Tor Vergata --
Dipartimento di Matematica --
Rome, Italy --
{\tt enrico.valdinoci@uniroma2.it}
}

\begin{abstract}
This paper is devoted to the study of rigidity properties for special solutions of nonlinear 
elliptic partial differential equations on smooth, boundaryless Riemannian manifolds. 
As far as stable solutions are concerned, we derive a new weighted Poincar\'e inequality
which allows to prove Liouville type results and the flatness of the level sets of the solution in dimension $2$, under 
suitable geometric assumptions on the ambient manifold.     
\end{abstract}
\maketitle
\tableofcontents
\section*{Notation}

Throughout this paper, $M$ will denote a complete,
connected,
smooth, $n$-dimensional,
manifold without boundary,
endowed with a smooth Riemannian metric $g=\{ g_{ij}\}$.

As customary, we consider the volume term induced by $g$,
that is, in local coordinates,
\begin{equation}\label{vg}
dV_g = \sqrt{|g|}\,dx^1\wedge \dots \wedge dx^n,
\end{equation}
where $\{ dx^1,\dots,dx^n\}$ is the basis of $1$-forms dual to the vector basis $\{\partial_i,\dots,\partial_n\}$,
and $|g|=\det(g_{ij})\ge0$.

We denote by $\div_g X$ the divergence of a smooth vector
field $X$ on $M$, that is, in local coordinates
$$ \div_g X=\frac{1}{\sqrt{|g|}} \partial_i \Big( \sqrt{|g|}X^i\Big),$$
where the Einstein summation convention is understood.

We also denote by $\nabla_g$ the Riemannian gradient
and by $\Delta_g$ the Laplace-Beltrami operator, that is,
in local coordinates,
$$ (\nabla_g\phi)^i=g^{ij}\partial_j \phi$$
and
\begin{equation}\label{Delta}
\Delta_g \phi= \div_g(\nabla_g \phi)
=
\frac{1}{\sqrt{|g|}}\partial_i \Big(
{\sqrt{|g|}} g^{ij}\partial_j \phi\Big)
,\end{equation}
for any smooth function $\phi:M\rightarrow \R$.

Due to this divergence structure (see, for example,
page~184 of~\cite{Gallot}), we have that
\begin{equation}\label{DIV}
\int_M \phi \Delta_g \psi \,dV_g = -\int_M \langle
\nabla_g \phi,\nabla_g \psi
\rangle\,dV_g,\end{equation}
for any
smooth~$\phi$, $\psi:M\rightarrow\R$, with either~$\phi$
or~$\psi$ compactly supported,
where $\langle \cdot,\cdot\rangle$ is the scalar product
induced by $g$ (no confusion should arise with the standard
Euclidean dot product).

In fact, by approximation, we have that~\eqref{DIV}
also holds when~$\phi$ is compactly supported and
Lipschitz continuous with respect to the metric structure
induced by~$g$.

Given a vector field $X$, we also denote
$$ |X|=\sqrt{\langle X,X\rangle}.$$

Also (see, for instance Definition~3.3.5 in~\cite{Jost}),
it is customary to define
the Hessian of a smooth function $\phi$ as
the symmetric $2$-tensor given in a local patch by
$$ (H_\phi)_{ij}=\partial^2_{ij}\phi-\Gamma^k_{ij}\partial_k\phi,$$
where $\Gamma^k_{ij}$ are the  
Christoffel symbols, namely
$$ 
\Gamma_{ij}^k=\frac12 g^{hk} \left( \partial_i g_{hj} +\partial_j g_{ih} -\partial_h g_{ij} \right) .$$
Given a tensor $A$,
we define its norm by $|A|=\sqrt{A A^*}$, where $A^*$
is the adjoint.

The above quantities are related to the Ricci tensor $\Ric_g$
via the Bochner-Weitzenb\"ock formula (see,
for instance,~\cite{Berger, Wang} and references therein):
\begin{equation}\label{BOC}
\frac 12\Delta_g |\nabla_g \phi|^2=
|H_\phi|^2+ \langle \nabla_g \Delta_g \phi,\nabla_g\phi\rangle
+\Ric_g (\nabla_g \phi,\nabla_g\phi).\end{equation}

We say that $M$ is parabolic if for any $p\in M$
there exists a precompact
neighborhood $U_p$ of $p$ in $M$
such that
for any $\epsilon>0$ there exists $\phi_\epsilon\in C^\infty_0
(M)$ for which $\phi_\epsilon (q)=1$ for any $q\in U_p$
and
\begin{equation}\label{PAR}
\int_{M} |\nabla \phi_\epsilon|^2\,dV_g\le\epsilon.\end{equation}
We refer to~\cite{Royden, Ly, Troy} for further comments
on parabolicity.

During the course of the paper, we will often use
normal coordinates at some fixed point $p_o\in M$
(see, for example, page~93 of~\cite{Gallot}); that is
we suppose that
\begin{equation}\label{NC}
g_{ij}(p_o) = \delta_{ij}, \qquad
{\partial_k g_{ij}}(p_o) = 0, \qquad{\mbox{
and
}}\qquad \Gamma^i_{jk}(p_o) = 0.\end{equation}

This paper will deal with solutions $u\in C^2(M)$ of
\begin{equation}\label{PDE}
- \Delta_g u=f(u),\end{equation}
where $f\in C^1(\R)$.

We say that a solution $u$ is stable if
\begin{equation}\label{STA}
\int_M |\nabla_g\xi|^2 -f'(u) \xi^2\,dV_g
\ge0\end{equation}
for every~$\xi\in C^\infty_0 (M)$.

Such a stability condition is customary in the calculus
of variations (see, for
example,~\cite{Moss, FCS, AAC}),
and it states that the second variation
of the (possibly formal) energy functional
associated to \eqref{PDE} is
nonnegative (for instance, local minima
of the energy are stable solutions).

\section{Main results}

We give the following Liouville type and flatness results:

\begin{theorem}\label{T1}
Let~$M$ be a connected Riemannian manifold.

Let $u$ be a stable solution of~\eqref{PDE}.
Suppose that 
\begin{itemize}
\item either $M$
is compact 
\item or $M$ is complete and
parabolic, and $|\nabla_gu|\in 
L^\infty(M)$.
\end{itemize}
Assume also
that the Ricci curvature is nonnegative
and that $\Ric_g$ does not vanish identically.

Then $u$ is constant.
\end{theorem}

Note that the conclusion of Theorem \ref{T1} is sharp. Indeed,  $\mathbb R^2$ endowed with its usual flat metric is parabolic (with identically zero Ricci tensor). The function 
$$u(x_1,x_2)=\tanh \left(\frac{x_1}{\sqrt{2}}\right) $$
is a stable non-constant solution of the two-dimensional
Allen-Cahn equation, namely
$$-\Delta u= u-u^3.$$

The previous example motivates the following result, which provides a rigidity property for stable solutions of \eqref{PDE} when $n=2$. 

\begin{theorem}\label{T3}
Let $M$ be a complete, connected Riemannian surface
(that is, a complete, connected
Riemannian manifold with~$\dim M=2$).

Let $u$ be a stable 
solution of 
\eqref{PDE}, with $|\nabla_g u|\in L^\infty(M)$. 
Assume also that $\Ric_g$ vanishes identically. Then, any
connected component of the 
level set of $u$ on which $\nabla_g u$ does not vanish is a geodesic. 
\end{theorem}

Of course, as well known,
in dimension~$n=2$, Ricci flat surfaces
are just surfaces with
zero Gaussian curvature, thence, in Theorem \ref{T3},
the assumption that~$\Ric_g$ vanishes identically
may be equivalently stated by requiring the Gaussian curvature
to vanish identically.

Also, Theorem \ref{T3} does not hold in high dimensions $n \geq 9$, as 
shown in 
\cite{PKW} for the Allen-Cahn equation in $\mathbb R^n $
endowed with its standard flat metric. More precisely, in $\mathbb 
R^9$ (with flat metric), one can construct monotone (hence stable,
see Corollary~4.3 in~\cite{AAC})
solutions whose level sets are 
not totally geodesic. 

This latter fact suggests that the parabolicity assumption in Theorem \ref{T3} (which is hidden in the two-dimensional character of $M$) seems to be necessary to obtain rigidity results on stable solutions of equation \eqref{PDE}. 

The proofs of our main results are based
on a geometric formula, which 
will be given in Theorem~\ref{SZ}
below, and which can be considered as a 
weighted Poincar\'e inequality. 

The use of such a type of formula in the 
Euclidean setting was started
in~\cite{SZ1, SZ2} and its importance for 
symmetry results was
explained in~\cite{FAR-H}. Further applications to 
PDEs have been given in~\cite{FSV, SV, FER}.

We now give two additional results in the spirit
of Theorem~\ref{T1},
under a sign assumption
on the nonlinearity and on the growth of the volume of the
geodesic balls.

For this, we denote~${\mathcal{B}}_R$
the (open)
geodesic ball of radius~$R> 0$, centered at
a given point of~$M$.

We denote by~${\mathcal{V}}_R$
the volume of~${\mathcal{B}}_R$,
computed with respect
to the volume element~$dV_g$ in~\eqref{vg}.

We obtain the 
following
results:

\begin{theorem}\label{ADDT1}
Let~$M$ be a complete, connected
Riemannian manifold
and let~$u$ be a bounded stable solution of 
\eqref{PDE}.

Suppose that
\begin{equation}\label{Si}
{\mbox{ $f(r)\ge 0$ for any~$r\in \R$}}
\end{equation}
and that \begin{equation}\label{vag}
\liminf_{R\rightarrow+\infty}
R^{-4} {\mathcal{V}}_{R}=0.
\end{equation}

Assume also
that the Ricci curvature of~$M$
is nonnegative
and that $\Ric_g$ does not vanish identically.

Then $u$ is constant.
\end{theorem}

\begin{theorem}\label{ADDT2}
Let~$M$ be a complete, connected
Riemannian manifold
and let~$u$ be a stable solution of~\eqref{PDE}.

Suppose that
\begin{equation}\label{33C}
\liminf_{R\rightarrow+\infty} R^{-2}{\mathcal{V}}_R
\Big( \sup_{
{\mathcal{B}}_R
} |\nabla_g u|\Big)^2 =0.
\end{equation}

Assume also
that the Ricci curvature of~$M$
is nonnegative
and that $\Ric_g$ does not vanish identically.

Then~$u$ is constant.
\end{theorem}

We recall that, for
complete, connected, $n$-dimensional
Riemannian manifolds with nonnegative Ricci curvature,
one controls~${\mathcal{V}}_R$ with~$R^{n}$ (see~\cite{Bishop}).
Therefore,~\eqref{vag} always holds when~$n\le 3$.

The paper is organized as follows. In~\S~\ref{SS1},
we make an observation about the positivity of
an interesting geometric quantity.
In~\S~\ref{SS2} we discuss
the weighted Poincar\'e inequality which will 
be the keystone
of the techniques presented here. {F}rom that,
useful flatness results are obtained in~\S~\ref{SS3}.

The proofs of the main results are given in~\S~\ref{SS4}--\ref{SS7}.

\section{A Preliminary result}\label{SS1}

{F}rom now on,~$M$ will always denote
a complete, connected
Riemannian manifold.

\begin{lemma}
For any smooth $\phi:M\rightarrow\R$, we have that
\begin{equation}\label{POS}
|H_\phi|^2\ge\big|\nabla_g|\nabla_g \phi|\big|^2\qquad
{\mbox{
almost everywhere.}}
\end{equation}

Also, equality holds at $p\in M\cap \{ \nabla_g \phi
\ne0\}$ if and 
only if
for any $k=1,\dots,n$ there exists $\kappa^k:M\rightarrow
\R$
such that
\begin{equation}\label{67}
\nabla_g\big( \nabla_g \phi\big)^k(p)=
\kappa^k(p) \nabla_g \phi(p).
\end{equation}
\end{lemma}

\begin{proof}
{F}rom 
Stampacchia's Theorem (see, for instance, Theorem~6.19
in~\cite{LOSS}), we know that~$\nabla_g |\nabla_g \phi|=0$
on~$\{ \nabla_g \phi=0\}$ up to a null-measure set.

Therefore, we can now concentrate on points in~$M\cap
\{\nabla_g
\phi\ne 0\}$.

Fix $p\in M
\cap\{\nabla_g
\phi\ne0\}
$, with $\nabla_g\phi(p)\ne0$.
Recalling \eqref{NC},
we use normal coordinates at $p$. Therefore~$(H_\phi)_{ij}(p)=
\partial^2_{ij}\phi(p)$ and so
$$ |H_\phi|^2 (p)= \sum_{1\le i,j\le n}
\Big( 
\partial^2_{ij}\phi(p)\Big).$$
Analogously, we have 
$$ |\nabla_g\psi\big|(p)=|\nabla\psi(p)|,$$
for any $\psi:M\rightarrow \R$ smooth in the vicinity
of~$p$. As a consequence,  
taking $\psi=|\nabla_g \phi|$, one gets 
\begin{eqnarray*}
&& \big|\nabla_g|\nabla_g \phi|\big|(p)=
\big|\nabla |\nabla_g \phi|\big|(p)\\ &&\qquad=
\left|\frac{\nabla_g \phi}{|\nabla_g \phi|}
\cdot \nabla(\nabla_g \phi)
\right|(p)
=\left|\frac{\nabla \phi}{|\nabla \phi|}
\cdot \nabla(\nabla_g \phi)
\right|(p)
.\end{eqnarray*}
Since, by~\eqref{NC},
$$ \partial_i \big( \nabla_g \phi)^h(p)=
\partial_i\big( g^{hk} \partial_k\phi)(p)=
\delta^{hk} \partial^2_{ik}\phi(p),$$
we thus obtain
$$\nabla\phi\cdot\nabla\big( \nabla_g \phi)^h(p)=
\sum_{1\le i\le n}
\partial_i \phi \partial^2_{ih}\phi(p).$$
Accordingly,
\begin{eqnarray*}
&& \big|\nabla_g|\nabla_g \phi|\big|^2(p)=
\frac{1}{|\nabla \phi|^2}
\sum_{1\le h\le n}\Big(
\partial_i \phi\partial^2_{ih}\phi
\Big)^2
=\frac{1}{|\nabla \phi|^2}
\sum_{1\le h\le n}\Big(
\nabla \phi\cdot \nabla (\partial_h \phi)
\Big)^2\\&&\qquad
\le
\sum_{1\le h\le n}\Big|
\nabla (\partial_h \phi)
\Big|^2=|H_\phi|^2 (p),
\end{eqnarray*}
with equality if and only if
$\nabla\phi$ and $\nabla(\partial_k\phi)$
are parallel, for any $k=1,\dots n$.
This gives the desired result.
\end{proof}

\section{A geometric inequality}\label{SS2}

\begin{theorem}\label{SZ}
Let $u$ be a stable solution of~\eqref{PDE}.
Then,
\begin{equation}\label{GF}
\int_M \Big(
\Ric_g (\nabla_gu,\nabla_gu)+|H_u|^2-\big|\nabla_g|\nabla_g u|\big|^2
\Big)\phi^2\,dV_g\le
\int_M |\nabla_g u|^2 |\nabla_g\phi|^2\,dV_g,\end{equation}
for any $\phi\in C^\infty_0(M)$.
\end{theorem}

\begin{proof} We take~$\phi\in
C^\infty_0(M)$ and $\xi=|\nabla_g 
u|\phi$
in~\eqref{STA} (note that this choice is possible in the light of
a density argument). We thus obtain
\begin{eqnarray*}
\int_M f'(u)|\nabla_gu|^2 \phi^2\,dV_g&\le&
\int_M \big|\nabla_g |\nabla_g u|\big|^2 \phi^2
+|\nabla_gu|^2|\nabla_g \phi|^2 +2\phi|\nabla_gu|
\langle\nabla_g\phi,\nabla_g|\nabla_g u|\rangle\,dV_g\\
&=&
\int_M \big|\nabla_g |\nabla_g u|\big|^2 \phi^2
+|\nabla_gu|^2|\nabla_g \phi|^2 +\frac12
\langle\nabla_g\phi^2,\nabla_g |\nabla_gu|^2\rangle\,dV_g
.\end{eqnarray*}
Therefore, recalling \eqref{DIV} and~\eqref{BOC},
\begin{eqnarray*}
\int_M f'(u)|\nabla_gu|^2 \phi^2\,dV_g&\le&
\int_M \big|\nabla_g |\nabla_g u|\big|^2 \phi^2
+|\nabla_gu|^2|\nabla_g \phi|^2 -\frac12
\phi^2\Delta_g |\nabla_gu|^2\,dV_g\\
&=&
\int_M \big|\nabla_g |\nabla_g u|\big|^2 \phi^2
+|\nabla_gu|^2|\nabla_g \phi|^2 \\&&\quad-
\phi^2
\Big(
|H_u|^2+ \langle \nabla_g \Delta_g u,\nabla_gu\rangle
+\Ric_g (\nabla_g u,\nabla_gu)
\Big)
\,dV_g
.\end{eqnarray*}
Since, by differentiating~\eqref{PDE}, we have that
$$ -\nabla_g \Delta_g u=f'(u)\nabla_g u,$$
we obtain
\begin{eqnarray*}
0&\le&
\int_M \big|\nabla_g |\nabla_g u|\big|^2 \phi^2
+|\nabla_gu|^2|\nabla_g \phi|^2 -
\phi^2
\Big(
|H_u|^2+\Ric_g (\nabla_g u,\nabla_gu)
\Big)
\,dV_g
,\end{eqnarray*}
which gives~\eqref{GF}.
\end{proof}

\section{Flatness lemmata}\label{SS3}

\begin{lemma}\label{Pre}
Let~$u$ be a smooth function on~$M$.

Assume that
\begin{equation}\label{544}
{\mbox{the Ricci curvature is nonnegative}}.\end{equation}

Suppose also that for any~$p\in M$
there exists a neighborhood $V_p$
of $p$ in $M$ such that
\begin{equation}\label{VP} \int_{V_p}
\Big(
\Ric_g (\nabla_gu,\nabla_gu)+|H_u|^2-\big|\nabla_g|\nabla_g u|\big|^2
\Big)\,dV_g\le 0.\end{equation}

Then,
\begin{equation}\label{77}
|H_u|^2(p)=\big|\nabla_g|\nabla_g u|\big|^2(p)
\qquad{\mbox{
for any $p\in M\cap\{\nabla_g u\ne0\}$,}}
\end{equation}
and
\begin{equation}\label{78}
\Ric_g (\nabla_gu,\nabla_gu) (p)=0\qquad{\mbox{
for any $p\in M$.}}
\end{equation} Furthermore,
for any $k=1,\dots,n$ there exist $\kappa^k:M\rightarrow
\R$
such that
\begin{equation}\label{79}
\nabla_g\big( \nabla_g u\big)^k(p)=
\kappa^k(p) \nabla_g u(p)\qquad{\mbox{
for any $p\in M\cap \{ \nabla_gu\ne0\}$.}}
\end{equation}
\end{lemma}

\begin{proof}
We fix $p\in M$ and we show that
\eqref{78} holds at $p$, 
and that~\eqref{77} and \eqref{79} hold at $p$ too
if~$\nabla_g u(p)\ne0$.

{F}rom~\eqref{POS}, \eqref{544}
and~\eqref{VP}, we have that
$$
\int_{V_p}
\Ric_g (\nabla_gu,\nabla_gu)
\,dV_g =0=\int_{V_p}
|H_u|^2-\big|\nabla_g|\nabla_g 
u|\big|^2
\,dV_g.$$
Accordingly,
\begin{equation}\label{01} \Ric_g (\nabla_gu,\nabla_gu)
=0
=|H_u|^2-\big|\nabla_g|\nabla_g
u|\big|^2\quad
{\mbox{almost everywhere in $V_p$.}}\end{equation}
Since~$\Ric_g$ is continuous,~\eqref{01}
implies
that~$\Ric_g (\nabla_gu,\nabla_gu)=0$ everywhere
in~$V_p$ and so~\eqref{78} holds at~$p$.

In addition, if~$p\in \{\nabla_g u\ne0\}$,
we have that the map~$|H_u|^2-\big|\nabla_g|\nabla_g
u|\big|^2$ is continuous in the vicinity of~$p$,
and so~\eqref{01} says that~\eqref{77} holds at~$p$
in this case.

Finally, \eqref{77} and \eqref{67} give \eqref{79}.
\end{proof}

\begin{lemma}\label{Dop}
Let $u$ be a stable solution of~\eqref{PDE}
and let the Ricci curvature of~$M$ be nonnegative.
Suppose that
\begin{itemize}
\item either $M$
is compact 
\item or $M$ is complete and parabolic, and $|\nabla_gu|\in 
L^\infty(M)$.
\end{itemize}
Then,~\eqref{77}, \eqref{78} and~\eqref{79}
hold true.
\end{lemma}

\begin{proof} 
We claim that there exists a neighborhood $V_p$
of $p$ in $M$ such that~\eqref{VP} holds.

Indeed,
if $M$ is compact
we can use Theorem~\ref{SZ}, by taking $\phi=1$ in~\eqref{GF},
obtaining
$$
\int_M 
\Ric_g (\nabla_gu,\nabla_gu)+|H_u|^2-\big|\nabla_g|\nabla_g u|\big|^2
\,dV_g\le
0.
$$
This gives \eqref{VP}, with $V_p=M$.

If, on the other hand, $M$ is parabolic and $|\nabla_gu|$ is bounded,
we fix $p\in M$, we recall Theorem~\ref{SZ} once more, we
take $\phi_\epsilon$ as in \eqref{PAR} and we plug it in 
\eqref{GF}: we recall~\eqref{POS}
and so we
obtain
\begin{eqnarray*}
&& \int_{U_p}
\Big(
\Ric_g (\nabla_gu,\nabla_gu)+|H_u|^2-\big|\nabla_g|\nabla_g u|\big|^2
\Big)\,dV_g
\\ &\le&
\int_M \Big(
\Ric_g (\nabla_gu,\nabla_gu)+|H_u|^2-\big|\nabla_g|\nabla_g u|\big|^2
\Big)\phi_\epsilon^2\,dV_g\\&\le&
\int_M |\nabla_g u|^2 |\nabla_g\phi_\epsilon|^2\,dV_g
\\ &\le& \|
\nabla_g u
\|_{L^\infty(M)}^2
\int_M  |\nabla_g\phi_\epsilon|^2\,dV_g
\\ &\le&\|
\nabla_g u
\|_{L^\infty(M)}^2 \epsilon.
\end{eqnarray*}
By taking $\epsilon$ arbitrarily small, we obtain~\eqref{VP},
with~$V_p=U_p$ in this case.

The desired result then follows from
Lemma~\ref{Pre}.
\end{proof}

\begin{lemma} \label{Fi1}
%% Let~$M$ be a complete, connected Riemannian manifold.
Suppose that the Ricci curvature of~$M$
is nonnegative and that~$\Ric_g$
does not vanish identically.

Let~$u$ be a solution
of~\eqref{PDE}, with
\begin{equation}\label{y7}{\mbox{$
\Ric_g (\nabla_gu,\nabla_gu) (p)=0$
for any $p\in M$.}}\end{equation}

Then,~$u$ is constant.
\end{lemma}

\begin{proof}
Since $\Ric_g$ is nonnegative definite and it does not vanish 
identically, 
we have that~$\Ric_g$ is positive definite
in a suitable open subset of~$M$.

Consequently,~\eqref{y7} implies 
that $\nabla_gu(p)=0$ for $p$ in a suitable open subset of~$M$.

Thence, $u$ is constant on such a subset. By
Unique Continuation
Principle (see Theorem 1.8 of \cite{Kazdan}), we have
that~$u$ is constant on~$M$.
\end{proof}

\section{Proof of Theorem \ref{T1}}\label{SS4}

{F}rom Lemma~\ref{Dop},
we have that~\eqref{78} holds true.

This makes it possible to use Lemma~\ref{Fi1}
and thus complete the proof of Theorem~\ref{T1}.

\section{Proof of Theorem \ref{T3}}

First of all, we observe that $M$ has nonnegative Gaussian curvature,
since it has nonnegative Ricci curvature and $\dim M=2$.
%%% http://en.wikipedia.org/wiki/Riemann_curvature_tensor

Therefore, from\footnote{We remark that we
are using here in a crucial way
the fact that $M$ has nonnegative Gauss 
curvature
to obtain \eqref{pro 79},
since there are examples of hyperbolic Riemannian surfaces (or, even,
hyperbolic two-dimensional graphs): see
\cite{Osserman1, Osserman2, Milnor}.} 
Theorem~15 of \cite{Huber}
(see also~\cite{P1,P2}), we get that
\begin{equation}\label{pro 79}{\mbox{
$M$ is parabolic.}}
\end{equation}
Take any connected component~${\mathcal{C}}$ of~$\{u=c\}\cap
\{ \nabla_gu\ne0\}$. Then,~${\mathcal{C}}$ is a 
smooth curve.

Thus, we take $\gamma:\R\rightarrow M$ to be~$
{\mathcal{C}}$ traveled with unit 
speed (with respect to the metric $g$),
that is
\begin{equation}\label{N1}
|\dot \gamma|^2 =1.\end{equation}
With this notation, Theorem \ref{T3} is proved once
we show that
\begin{equation}\label{GEO}
\ddot\gamma^k
+ \Gamma^{k}_{ij}\dot\gamma^i
\dot\gamma^j = 0. \end{equation}
To prove \eqref{GEO}, we take any $t_o\in \R$ and we show
that~\eqref{GEO}
holds at $t_o$. For this, we choose a normal 
coordinate
frame at $p_o=\gamma(t_o)$.

Then, from \eqref{N1},
\begin{eqnarray*} 0&=& \frac12\,
\frac{d}{dt}\Big(g_{ij}(\gamma(t)) \dot\gamma^i(t)
\dot\gamma^j(t)\Big)\\&=&
\frac12 \partial_k g_{ij} (\gamma(t))\dot\gamma^k(t)
\dot\gamma^i(t)
\dot\gamma^j(t)
+ g_{ij}(\gamma(t))\dot\gamma^i(t)
\ddot\gamma^j(t).
\end{eqnarray*}
Consequently, from \eqref{NC}, we have 
\begin{equation}\label{3322}
0= \dot \gamma(t_o)\cdot \ddot\gamma(t_o).
\end{equation}
Moreover, since $u(\gamma(t))=c$, we also have
\begin{equation}\label{1167}
0=\frac{d}{dt} u(\gamma(t))=
\partial_i u(\gamma(t)) \dot\gamma^i(t).
\end{equation}
By differentiating \eqref{1167}
once more time, one gets
\begin{equation}\label{5611}
0=\frac{d}{dt}\Big(
\partial_i u(\gamma(t)) \dot\gamma^i(t)\Big)\\
=
\partial_{ij} u(\gamma(t)) \dot\gamma^i(t)\dot\gamma^j(t)
+\partial_i u(\gamma(t)) \ddot\gamma^i(t).
\end{equation}
We now observe that \eqref{pro 79} and Lemma~\ref{Dop}
make
it possible to use \eqref{79} here. 

Accordingly, from \eqref{79} 
and \eqref{NC}
we obtain, for any $j=1,\dots,n$,
$$
\partial_j  \nabla u(p_o)=
\kappa_j(p_o) \nabla u(p_o),$$
for some $\kappa_j(p_o)\in\R$.

This and \eqref{5611} give that
\begin{eqnarray*}
0&=&
\partial_{ij} u(p_o) \dot\gamma^i(t_o)\dot\gamma^j(t_o)
+\partial_i u(p_o) \ddot\gamma^i(t_o)\\&=&
\kappa_j(p_o) \partial_i u(p_o)
\dot\gamma^i(t_o)\dot\gamma^j(t_o)
+\partial_i u(p_o) \ddot\gamma^i(t_o)
\\ &=&
\Big(\kappa_j(p_o) \dot\gamma^j(t_o)\Big)
\,\Big(\partial_i u(p_o)
\dot\gamma^i (t_o)\Big)
+\partial_i u(p_o) \ddot\gamma^i(t_o)
.\end{eqnarray*}
Hence, employing \eqref{1167},
\begin{equation}\label{ER2}
0 =\partial_i u(p_o) \ddot\gamma^i(t_o).
\end{equation}
By \eqref{3322} and \eqref{ER2},
we see that $\ddot \gamma (t_o)$ is orthogonal
(in the Euclidean sense) 
both to $\dot\gamma(t_o)$,
which is tangent to
$\{u=c\}$ at $p_o$, and
to $\nabla u(p_o)$, which is normal to
$\{u=c\}$ at $p_o$.

Therefore, $\ddot\gamma(t_o)=0$.

As a consequence, from \eqref{NC},
$$
\ddot\gamma^k(t_o)
+ \Gamma^{k}_{ij}(p_o) \dot\gamma^i(t_o)
\dot\gamma^j (t_o)
=\ddot\gamma^k(t_o)+0=0
.$$
This proves \eqref{GEO} at the generic time $t=t_o$
and it thus completes the proof
of Theorem \ref{T3}.

\section{A useful cutoff}

For the proof of Theorems~\ref{ADDT1} and~\ref{ADDT2},
it is useful to introduce the following cutoff function.

Let~$d_g$ be the geodesic distance. Then~${\mathcal{B}}_R=
\{ p\in M {\mbox{
s.t.
}} d_g(p) < R\}$.

Fix~$\tau\in C^\infty_0 ([-2,2],
[0,1])$ with~$\tau(t)=1$
for any~$t\in [-1,1]$.

Given~$R>0$, for any~$p\in M$, we define
\begin{equation}\label{deca0}
\tau_R (p)=\tau \left(\frac{d_g (p)}{R}\right).
\end{equation}
Then,\begin{equation}\begin{split}\label{deca}
&{\mbox{
$\tau_R(p)=1$ for any $p\in{\mathcal{B}}_R$,
$\tau_R(p)=0$ for any $p\in M\setminus
{\mathcal{B}}_{2R}$,
and}}\\&\qquad |\nabla_g \tau_R (p)|\le
\frac{C_o}{R} \chi_{
{\mathcal{B}}_{2R}\setminus
{\mathcal{B}}_R
}(p)\qquad
{\mbox{
for any
$p\in M$.
}}\end{split}
\end{equation}

Then, we have:

\begin{lemma}\label{zac}
%% Let~$M$ be a complete, connected
%% Riemannian manifold.
Suppose that the Ricci curvature of~$M$
is nonnegative and that~$\Ric_g$
does not vanish identically.

Let~$u$ be a stable
solution
of~\eqref{PDE} such that
\begin{equation}\label{8kkj}
\liminf_{R\rightarrow+\infty}\int_M
|\nabla_gu|^2|\nabla_g\tau_R|^2\,dV_g=0.
\end{equation}
Then $u$ is constant.
\end{lemma}

\begin{proof} {F}rom~\eqref{POS},
\eqref{deca}, \eqref{8kkj} and
Theorem~\ref{SZ},\begin{eqnarray*}
&&\int_M \Big(
\Ric_g (\nabla_gu,\nabla_gu)+|H_u|^2-\big|\nabla_g|\nabla_g
u|\big|^2\Big)\,dV_g \\&=&
\liminf_{R\rightarrow+\infty}
\int_{\mathcal{B}_R}\Big(
\Ric_g (\nabla_gu,\nabla_gu)+|H_u|^2-\big|\nabla_g|\nabla_g
u|\big|^2\Big)\,dV_g\\ &\le&
\liminf_{R\rightarrow+\infty}
\int_{M}\Big(
\Ric_g (\nabla_gu,\nabla_gu)+|H_u|^2-\big|\nabla_g|\nabla_g
u|\big|^2\Big)\tau_R^2 \,dV_g\\
&\le&\liminf_{R\rightarrow+\infty}
\int_M |\nabla_g u|^2 |\nabla_g\tau_R|^2\,dV_g
\\&=&0.\end{eqnarray*}
Hence,~\eqref{VP} holds true with~$V_p=M$.

Therefore, by Lemma~\ref{Pre},~$\Ric_g(\nabla_gu,\nabla_gu)$
vanishes identically on~$M$.

Hence, the desired result follows from
Lemma~\ref{Fi1}.
\end{proof}

\section{Proof of Theorem~\ref{ADDT1}}

Let~$m_-$, $m_+\in\R$ be such that~$m_-\le u(p)\le m_+$
for any~$p\in M$. Let also~$\tau_R$
as in~\eqref{deca0}.

Making use of~\eqref{DIV}
and~\eqref{Si},
we see that
\begin{eqnarray*}
0 &\ge&\int_M f(u) (u-m_+) \tau_R^2 \,dV_g\\
&=& \int_M
\langle
\nabla_g u, \nabla_g \big( (u-m_+)\tau_R^2
\big)
\rangle
\,dV_g
\\ &=& 
\int_{\mathcal{B}_{2R}} |\nabla_g u|^2 \tau_R^2\,dV_g
+2
\int_{\mathcal{B}_{2R}} \langle \nabla_g u,\nabla_g
\tau_R\rangle \tau_R (u-m_+) \,dV_g
\\ &\ge&
\int_{\mathcal{B}_{2R}} |\nabla_g u|^2 \tau_R^2\,dV_g
- 2(m_+-m_-)
\int_{\mathcal{B}_{2R}} | \nabla_g u|\, |\nabla_g
\tau_R|\, \tau_R  \,dV_g.
\end{eqnarray*}
Therefore, by Cauchy-Schwarz inequality,
$$ 0\ge \frac12
\int_{\mathcal{B}_{2R}} |\nabla_g u|^2 \tau_R^2\,dV_g
- C_\star
\int_{\mathcal{B}_{2R}} |\nabla_g
\tau_R|^2  \,dV_g$$
for a suitable~$C_\star>0$, possibly depending on~$m_-$ and~$m_+$,
and so, recalling~\eqref{deca},
\begin{equation}\label{98990}
\begin{split}
&\int_{\mathcal{B}_{R}} |\nabla_g u|^2 \,dV_g\le
\int_{\mathcal{B}_{2R}} |\nabla_g u|^2 \tau_R^2\,dV_g
\\&\qquad
\le 2
C_\star
\int_{\mathcal{B}_{2R}} |\nabla_g
\tau_R|^2  \,dV_g
\le \frac{\bar C}{R^2} {\mathcal{V}}_{2R},
\end{split}\end{equation}
for some~$\bar C>0$ which does not depend on~$R$.

{F}rom~\eqref{vag}, \eqref{deca}
and~\eqref{98990},
we conclude that
$$ \liminf_{R\rightarrow+\infty}
\int_M |\nabla_g u|^2 |\nabla_g \tau_R|^2
\,dV_g \le 
\liminf_{R\rightarrow+\infty}
\frac{C_o^2}{R^2}
\int_{{\mathcal{B}}_{2R}}
|\nabla_g u|^2 
\,dV_g\le
\liminf_{R\rightarrow+\infty}
\frac{C_o^2 \bar C
}{ 4 R^4} {\mathcal{V}}_{4R}
=0.$$
Then, we use Lemma~\ref{zac} to
end the proof of Theorem~\ref{ADDT1}.

\section{Proof of Theorem~\ref{ADDT2}}\label{SS7}

We take~$\tau_R$ as in~\eqref{deca0}. Then, 
from~\eqref{33C} and~\eqref{deca},
\begin{eqnarray*}
\liminf_{R\rightarrow+\infty} 
\int_M |\nabla_g u|^2 |\nabla_g\tau_R|^2\,dV_g
&=&
\liminf_{R\rightarrow+\infty}
\int_{\mathcal{B}_{2R}} |\nabla_g u|^2 |\nabla_g\tau_R|^2\,dV_g
\\
&\le&\liminf_{R\rightarrow+\infty} 
\Big( \sup_{
{\mathcal{B}_{2R}}
} |\nabla_g u|\Big)^2 \,\frac{C_o^2}{R^2}
\, \int_{\mathcal{B}_{2R}}\,dV_g\\
&=&0.
\end{eqnarray*}
Then, the proof of Theorem~\ref{ADDT2}
is ended via Lemma~\ref{zac}.

\bibliographystyle{alpha}
\bibliography{bibman}

\def\cprime{$'$}
\begin{thebibliography}{dPKW08}

\bibitem[AAC01]{AAC}
Giovanni Alberti, Luigi Ambrosio, and Xavier Cabr{\'e}.
\newblock On a long-standing conjecture of {E}. {D}e {G}iorgi: symmetry in 3{D}
  for general nonlinearities and a local minimality property.
\newblock {\em Acta Appl. Math.}, 65(1-3):9--33, 2001.
\newblock Special issue dedicated to Antonio Avantaggiati on the occasion of
  his 70th birthday.

\bibitem[BC64]{Bishop}
Richard~L. Bishop and Richard~J. Crittenden.
\newblock {\em Geometry of manifolds}.
\newblock Pure and Applied Mathematics, Vol. XV. Academic Press, New York,
  1964.

\bibitem[BGM71]{Berger}
Marcel Berger, Paul Gauduchon, and Edmond Mazet.
\newblock {\em Le spectre d'une vari\'et\'e riemannienne}.
\newblock Lecture Notes in Mathematics, Vol. 194. Springer-Verlag, Berlin,
  1971.

\bibitem[CY75]{P1}
S.~Y. Cheng and S.~T. Yau.
\newblock Differential equations on {R}iemannian manifolds and their geometric
  applications.
\newblock {\em Comm. Pure Appl. Math.}, 28(3):333--354, 1975.

\bibitem[dPKW08]{PKW}
Manuel del Pino, Mike Kowalczyk, and Juncheng Wei.
\newblock On {D}e {G}iorgi {C}onjecture in {D}imension ${N} \geq 9$.
\newblock {\em {P}reprint}, 2008.

\bibitem[Far02]{FAR-H}
Alberto Farina.
\newblock Propri\'et\'es qualitatives de solutions d'\'equations et syst\`emes
  d'\'equations non-lin\'eaires.
\newblock 2002.
\newblock Habilitation \`a diriger des recherches, Paris VI.

\bibitem[FCS80]{FCS}
Doris Fischer-Colbrie and Richard Schoen.
\newblock The structure of complete stable minimal surfaces in {$3$}-manifolds
  of nonnegative scalar curvature.
\newblock {\em Comm. Pure Appl. Math.}, 33(2):199--211, 1980.

\bibitem[FSV08]{FSV}
Alberto Farina, Berardino Sciunzi, and Enrico Valdinoci.
\newblock {B}ernstein and {D}e {G}iorgi type problems: new results via a
  geometric approach.
\newblock {\em {T}o appear in {A}nn. {S}cuola {N}orm. {S}up. {P}isa {C}l.
  {S}ci.}, 2008.

\bibitem[FV08]{FER}
Fausto Ferrari and Enrico Valdinoci.
\newblock A geometric inequality in the heisenberg group and its applications
  to stable solutions of semilinear problems.
\newblock {\em {T}o appear in {M}ath. {A}nn.}, 2008.

\bibitem[GHL90]{Gallot}
Sylvestre Gallot, Dominique Hulin, and Jacques Lafontaine.
\newblock {\em Riemannian geometry}.
\newblock Universitext. Springer-Verlag, Berlin, second edition, 1990.

\bibitem[GT99]{Troy}
Vladimir Gol{\cprime}dshtein and Marc Troyanov.
\newblock The {K}elvin-{N}evanlinna-{R}oyden criterion for {$p$}-parabolicity.
\newblock {\em Math. Z.}, 232(4):607--619, 1999.

\bibitem[Hub57]{Huber}
Alfred Huber.
\newblock On subharmonic functions and differential geometry in the large.
\newblock {\em Comment. Math. Helv.}, 32:13--72, 1957.

\bibitem[Jos98]{Jost}
J{\"u}rgen Jost.
\newblock {\em Riemannian geometry and geometric analysis}.
\newblock Universitext. Springer-Verlag, Berlin, second edition, 1998.

\bibitem[Kaz88]{Kazdan}
Jerry~L. Kazdan.
\newblock Unique continuation in geometry.
\newblock {\em Comm. Pure Appl. Math.}, 41(5):667--681, 1988.

\bibitem[LL97]{LOSS}
Elliott~H. Lieb and Michael Loss.
\newblock {\em Analysis}, volume~14 of {\em Graduate Studies in Mathematics}.
\newblock American Mathematical Society, Providence, RI, 1997.

\bibitem[LS84]{Ly}
Terry Lyons and Dennis Sullivan.
\newblock Function theory, random paths and covering spaces.
\newblock {\em J. Differential Geom.}, 19(2):299--323, 1984.

\bibitem[Mil77]{Milnor}
John Milnor.
\newblock On deciding whether a surface is parabolic or hyperbolic.
\newblock {\em Amer. Math. Monthly}, 84(1):43--46, 1977.

\bibitem[MP78]{Moss}
William~F. Moss and John Piepenbrink.
\newblock Positive solutions of elliptic equations.
\newblock {\em Pacific J. Math.}, 75(1):219--226, 1978.

\bibitem[Oss56a]{Osserman1}
Robert Osserman.
\newblock A hyperbolic surface in {$3$}-space.
\newblock {\em Proc. Amer. Math. Soc.}, 7:54--58, 1956.

\bibitem[Oss56b]{Osserman2}
Robert Osserman.
\newblock Riemann surfaces of class {${\rm A}$}.
\newblock {\em Trans. Amer. Math. Soc.}, 82:217--245, 1956.

\bibitem[Roy52]{Royden}
H.~L. Royden.
\newblock Harmonic functions on open {R}iemann surfaces.
\newblock {\em Trans. Amer. Math. Soc.}, 73:40--94, 1952.

\bibitem[SV08]{SV}
Yannick Sire and Enrico Valdinoci.
\newblock Fractional {L}aplacian and boundary reactions phase trnasitions: a
  geometric inequality and a symmetry result.
\newblock {\em {P}reprint}, 2008.

\bibitem[SZ98a]{SZ1}
Peter Sternberg and Kevin Zumbrun.
\newblock Connectivity of phase boundaries in strictly convex domains.
\newblock {\em Arch. Rational Mech. Anal.}, 141(4):375--400, 1998.

\bibitem[SZ98b]{SZ2}
Peter Sternberg and Kevin Zumbrun.
\newblock A {P}oincar\'e inequality with applications to volume-constrained
  area-minimizing surfaces.
\newblock {\em J. Reine Angew. Math.}, 503:63--85, 1998.

\bibitem[Var81]{P2}
Nicholas~Th. Varopoulos.
\newblock The {P}oisson kernel on positively curved manifolds.
\newblock {\em J. Funct. Anal.}, 44(3):359--380, 1981.

\bibitem[Wan05]{Wang}
Jiaping Wang.
\newblock Lecture notes on geometric analysis.
\newblock 2005.
\newblock {\tt www.math.nthu.edu.tw/user/writing/w1168.pdf}.

\end{thebibliography}

\end{document}